\documentclass[a4paper,11pt]{amsart}

\usepackage{color}

\theoremstyle{plain}
\newtheorem{thm}{Theorem}[section]

\newtheorem{lemma}[thm]{Lemma}

\theoremstyle{definition}

\newtheorem{example}[thm]{Example}

\numberwithin{equation}{section}

\newcommand{\sB}{{\mathcal B}}

\newcommand{\sR}{{\mathcal R}}

\newcommand{\PP}{\ensuremath{\mathbb{P}}}

\newcommand{\ZZ}{\ensuremath{\mathbb{Z}}}

\newcommand{\hol}{\ensuremath{\mathcal{O}}}

\newcommand\Lam{\Lambda}

\newcommand\Ga{\Gamma}

\newcommand\ga{\gamma}
\newcommand\de{\delta}

\newcommand{\ra}{\ensuremath{\rightarrow}}

\def\eea{\end{eqnarray*}}
\def\bea{\begin{eqnarray*}}

\newcommand\dual{\mathrel{\raise3pt\hbox{$\underline{\mathrm{\thinspace d
\thinspace}}$}}}
\newcommand\qe{\ifhmode\unskip\nobreak\fi\quad $\Box$}       

\def\BOX{\hfill\lower.5\baselineskip\hbox{$\Box$}}

\newtheorem{theo}{Theorem}[section]

\newtheorem{remarkk}[theo]{Remark}
\newenvironment{rem}{\begin{remarkk}\rm}{\end{remarkk}}

\newtheorem{conj}[theo]{Conjecture}

\newcommand{\core}[1]{core(#1)}      

\newenvironment{dedication}
        {\begin{quotation}\begin{center}\begin{em}}
        {\par\end{em}\end{center}\end{quotation}}

\title [Canonical Maps of  Hypersurfaces in Abelian Varieties ]{Canonical Maps of general     Hypersurfaces in Abelian Varieties}
\author{Fabrizio Catanese 
and Luca Cesarano
}
\address{Lehrstuhl Mathematik VIII, Mathematisches Institut der Universit\"{a}t
Bayreuth, NW II, Universit\"{a}tsstr. 30,
95447 Bayreuth}
\email{Fabrizio.Catanese@uni-bayreuth.de}
\email{Luca.Cesarano@uni-bayreuth.de}

\thanks{
The present work took place in the framework  of the 
 ERC Advanced grant n. 340258, `TADMICAMT'  }

\date{\today}

\begin{document}

\begin{abstract}
The   main theorem of this paper  is that,   for a general  pair $(A,X)$ of an (ample) hypersurface $X$ in an Abelian Variety $A$,
the canonical map  $\Phi_X$ of $X$   is birational onto its image if the polarization given by $X$ is not principal (i.e., its Pfaffian $d$ is not equal to $1$). 

We also easily show that, setting $g = dim (A)$, and letting $d$ be the Pfaffian of the polarization given by $X$, then if $X$ is smooth
and $$\Phi_X : X \ra \PP^{N:=g+d-2}$$ is an embedding, then necessarily  we have the inequality
$ d \geq g + 1$, equivalent to $N : = g+d-2 \geq 2 \ dim(X) + 1.$

Hence we   formulate the following interesting  conjecture, motivated by work of the second author:  if $ d \geq g + 1,$ then, for a general  pair $(A,X)$,  $\Phi_X$ is an embedding.

\end{abstract}

\maketitle
\begin{dedication}
Dedicated to Olivier Debarre on the occasion of his 60-th  $+ \epsilon$
 birthday.
\end{dedication}



\section {Introduction}

Let $A$ be an Abelian variety of dimension $g$, and let $X \subset A$ be a smooth ample hypersurface in $A$ such that the Chern class $c_1(X)$ of
the divisor $X$ is a polarization of type $\overline{d} : = (d_1, d_2, \dots, d_{g})$, so that the vector space $H^0(A, \hol_A(X))$
has dimension equal to the Pfaffian $ d : = d_1 \cdot \dots \cdot d_g$ of $c_1(X)$. 

The classical results of Lefschetz \cite{lefschetz} say that the rational map associated to $H^0(A, \hol_A(X))$ is a morphism if $d_1\geq 2$, and is an embedding of $A$ if $d_1 \geq 3$. 

There have been several improvements in this direction, by work of several authors, for instance \cite{ohbuchi} showed that for $d_1\geq 2$
we have an embedding except in a very special situation, and, for  progress in the case $d_1 =1$, see for instance \cite{ramanan}, \cite{debarreetal}.

Now, by adjunction, the canonical sheaf  of $X$ is the restriction $\hol_X(X)$, so a natural generalization of Lefschetz' theorems
is to ask about the behaviour of the canonical systems of such hypersurfaces $X$. Such a question is important in birational geometry, but  the results  for the canonical maps can depend on the hypersurface $X$ and not just  on  the polarization type only. 
 
We succeed in this paper to find (respectively: conjecture) simple results
for general such hypersurfaces.

Our work  was motivated by a theorem  obtained by the first author in a joint work with Schreyer \cite{c-s} on canonical surfaces:
if we have a polarization of type $(1,1,2)$ then the image $\Sigma$ of the canonical map $\Phi_X$ is in general a surface of degree $12$ in $\PP^3$,
birational to $X$, while for the special case where $X$ is the pull-back of the Theta divisor of 
a curve of genus $3$, then the canonical map has degree $2$, and $\Sigma$ has degree $6$.

The connection of the  above  result with the Lefschetz theorems is, as we already said,  provided by adjunction, we have 
 the following folklore result, a proof of which can be found for instance  in  \cite{cesarano}
 (a referee pointed out that the proof in the case of a principal polarization appears in 2.10 of \cite{green},
 and that of course Green's  proof works in general)

\begin{lemma}\label{canmap}
Let $X$ be a smooth ample hypersurface of dimension $n$ in an Abelian variety $A$, such that the class of $X$ is
a polarization of type  $\overline{d} : = (d_1, d_2, \dots, d_{n+1})$.

Let $\theta_1, \dots, \theta_d$ be a basis of $H^0(A, \hol_A(X))$ such that $X = \{ \theta_1=0\}$. 

If $z_1, \dots , z_g$ are linear  coordinates on the complex vector space $V$ such that $A$ is the quotient of
$V$ by a lattice $\Lam$, $A = V/ \Lam$, then the  canonical map $\Phi_X$ is given by 
$$(\theta_2, \dots, \theta_d,\frac{ \partial \theta_1}{ \partial z_1},\dots,  \frac{ \partial \theta_1}{ \partial z_g}).$$ 
\end{lemma}

Hence first of all the canonical map is an embedding if $H^0(A, \hol_A(X))$ yields an embedding of $A$; secondly, 
since a projection of $\Phi_X$ is the Gauss map of $X$, given by 
$(\frac{ \partial \theta_1}{ \partial z_1},\dots,  \frac{ \partial \theta_1}{ \partial z_g}),$
by a theorem of Ziv Ran \cite{ran} it follows that 
 the canonical system $|K_X|$ is base-point-free and 
$\Phi_X$  a  finite morphism. 
 
 This is our main result:

\begin{thm}\label{genbirat}
Let $(A,X)$ be a general pair, consisting of a hypersurface $X$ of dimension $n = g-1$ in an Abelian variety $A$, such that 
the class of $X$ is
a polarization of type  $\overline{d} : = (d_1, d_2, \dots, d_{g})$ with Pfaffian $ d = d_1\dots d_g > 1$.

Then the canonical map $\Phi_X $ of $X$ is birational onto its image $\Sigma$.
\end{thm}

The first observation  is:   the hypothesis that we  take  a general such pair, and not any pair,
is necessary in view of the    cited result of \cite{c-s}.

The second observation is that  the above result extends to more general situations, using a result on openness
of birationality (this will be pursued elsewhere).  This allows another proof of the theorem,  obtained studying pull-backs of Theta divisors of hyperelliptic curves (observe that for Jacobians the Gauss map of the Theta divisor
is a rational map, see \cite{cgs} for a study of its degree).

Here  we use the following 
nice result by Olivier Debarre \cite{debarre}:

 \begin{thm}\label{debarre}
Let $\Theta \subset A'$ be the  Theta divisor of a general principally polarized Abelian variety $A'$: then the Gauss map of $\Theta$,
$\psi : \Theta \ra P : = \PP^{g-1}$ factors through $Y : = \Theta / \pm 1$, and yields a generic covering $\Psi : Y \ra P$,
meaning that 

\begin{enumerate}
\item
The branch divisor $\sB \subset P $ of $\Psi$ is irreducible and reduced;
\item
The ramification divisor of $\Psi$, $\sR \subset Y$, maps birationally to $\sB$;
\item
the local monodromies of the covering at the general points of $\sB$ are transpositions;
\item
the Galois group of $\Psi$ (the global monodromy group of the covering $\Psi$) is the symmetric group $\mathfrak S_N$,
where $ N = (\frac{1}{2} g!)$.

\end{enumerate}
\end{thm}

The next question to which  the previous result paves the way is: when is $\Phi_X$ an embedding for $(A,X)$ general? 

 An elementary application of the  Severi double point formula \cite{Severi} (see also \cite{FL77}, \cite{Ca79}) as an embedding obstruction, yields  a necessary condition  (observe that a similar argument was used by van de Ven in \cite{vdv}, in order to study
the embeddings of Abelian varieties).

 \begin{thm}\label{obstruction}
Let $X\subset A$ be a smooth  ample hypersurface in an Abelian variety of dimension $g$, giving a polarization with Pfaffian $d$.

If the canonical map $\Phi_X$ is an embedding of $X$, then necessarily 
$$ d \geq g + 1.$$

\end{thm}

With some optimism  (hoping  for  a simple result), but relying on the highly non-trivial positive result of the second author 
\cite{cesarano} concerning polarizations of type $(1,2,2)$
(here $g=3$, $d=4$), and on the result by Debarre and others  \cite{debarreetal} that this holds true for
polarizations of type $(1,\dots, 1, d)$ with $ d > 2^g$,  we pose the following conjecture:

\begin{conj}\label{main}
Assume that $(A,X)$ is a general pair of  a smooth  ample hypersurface $X\subset A$  in an Abelian variety $A$, giving a polarization with Pfaffian $d$
satisfying $ d \geq g + 1.$

Then  the canonical map $\Phi_X$ is an embedding of $X$.

\end{conj}

We end the paper discussing the conjecture.

\section{Proof of the main theorem \ref{genbirat} }

We give first a quick outline of the strategy of the proof.

The first step 2.1 reduces to the case where the Pfaffian $d$ of the polarization type is a prime number $p$.

Step 2.2 considers the particular case where  $d=p$ and  where $X$ is the pull-back of the Theta divisor $\Theta$ of
a principally polarized Abelian variety. In this case the Gauss map of $X$ factors through $Y: =  \Theta / \pm 1$.

Step 2.3 shows that if $g = dim (A)=2$, then Step 2.2 works and the theorem is proven.

Step 2.4, the Key Step, shows that if $ g \geq 3$ and 2.2 does not work, then the image $\Sigma$ of the canonical map
lies birationally between $X$ and $Y$, hence corresponds to a subgroup of the dihedral group $D_p$

Step 2.5 finishes the proof, showing that, in each of the two possible cases corresponding to the  subgroups
of $D_p$, the canonical map becomes birational onto its image for a general deformation of $X$.

\subsection{Reduction to the case of a polarization of type  $(1,\dots, 1, p)$,
with $p$ a prime number.}
We shall proceed by induction, basing on the following concept.

We shall say that a polarization type $\overline{d} : = (d_1, d_2, \dots, d_{g})$ is divisible by
$\overline{\de} : = (\de_1, \de_2, \dots, \de_{g})$ if, for all $i=1, \dots, g$, we have that $\de_i$ divides
$d_i$, $ d_i = r_i \de_i$.

\begin{lemma}\label{induction}
Assume that the polarization type $\overline{d} $ is divisible by
$\overline{\de} $. Then, if the main Theorem \ref{genbirat} holds true for $\overline{\de}$, then it also holds
true  for $\overline{d}$.
\end{lemma}
\begin{proof}
We let $(A', X')$ be a general pair giving a polarization of type $\overline{\de}$, so that $\Phi_{X'}$ is birational. 

There exists an \'etale
covering $ A \ra A'$, with kernel $ \cong \oplus_i (\ZZ/ r_i)$, such that the pull back of $X'$ is a polarization of type
$\overline{d}$. 

By induction, we may assume without loss of generality that all numbers $r_i=1$, with one exception $r_j$, which is a prime number $p$.

Then we have $\pi : X \ra X'$, which is an \'etale quotient with group $\ZZ/p$.  Moreover, since the  canonical divisor $K_X$
is the pull-back
of $K_{X'}$, the composition $\Phi_{X'} \circ \pi$ factors through $\Phi_X$, and a morphism $f : \Sigma \ra \Sigma'$.

Since by assumption  $\Phi_{X'}$ is birational, either $\Phi_X$ is birational, and there is nothing to prove,
or $\Phi_X$ has degree $p$, $f : \Sigma \ra \Sigma'$ is birational, hence $\Phi_X$ factors (birationally) through $\pi$.

To contradict the second alternative, it suffices to show that the canonical system $H^0 (X, \hol_X (K_X))$ separates the points
of a general fibre of $\pi$.  Since each fibre is an orbit for the group $\ZZ/p$, it suffices to show that there are at least two different eigenspaces
in the vector space $H^0 (X, \hol_X (K_X))$, with respective eigenvalues $1$,  $\zeta$, where $\zeta$ is a primitive  $p$-th root of unity.

Since then we would have as projection of the canonical map a  rational map  $F : X \dashrightarrow \PP^1$ 
such that, for $g \in \ZZ/p$,  $$ F(x) = (x_0, x_1) \neq F(gx) = (x_0, \zeta x_1) ,$$
thereby separating the points of a general fibre.

Now, if there were only  one  non-zero eigenspace, the one  for the  eigenvalue $1$,  since we have an eigenspace decomposition 
$$H^0 (X, \hol_X (K_X))= H^0 (X,\pi^* ( \hol_{X'} (K_{X'})) =    H^0 (X',\pi_* \pi^* ( \hol_{X'} (K_{X'}))= $$ 
$$ = \oplus_1^p H^0 (X', \hol_{X'} (K_{X'} + i \eta)),$$
(here $\eta$ is a nontrivial divisor of $p$-torsion), 
all eigenspaces would have dimension zero, except for the case $i=p$. In particular   we would have that
$H^0 (X, \hol_X (K_X))$ and $H^0 (X', \hol_{X'} (K_{X'} ))$ have the same dimension.

 But this is a contradiction,
 since  the dimension $h^0 (X, \hol_X (K_X))$ equals by lemma \ref{canmap}  the sum of the Pfaffian of $X$ with $g-1$,
 and the Pfaffian of $X$ is $p$-times the Pfaffian of $X'$.
 
\end{proof}

\subsection{The special case where $X$ is a pull-back of a Theta divisor.}
Here, we shall consider  a similar situation, assuming that $X$ is a polarization of type  $(1,\dots, 1, p)$,
and that $X$ is the pull back of a Theta divisor $\Theta \subset A'$ via an isogeny $A \ra A'$ with kernel $\cong \ZZ/p$.

We define $Y : = \Theta / \pm 1$ and use the cited result of Debarre \cite{debarre}.

We consider the Gauss map of $X$, $f  : X \ra P : = \PP^{g-1}$, and observe that we have a factorization 
$$ f = \Psi \circ \phi, \ \phi : X \ra Y , \ \Psi : Y \ra P.$$
The essential features are that:

(i) $\phi$ is a Galois quotient $ Y = X /G$, where $G$ is the dihedral group $D_p$.

(ii) $\Psi$ is a generic map, in particular with monodromy group equal to $\mathfrak S_N$, $ N = g! / 2$.

Either the theorem is true, or, by contradiction, we have a factorization of $f$ through $\Phi_X : X \ra \Sigma$,
whose  degree $m$ satisfies $ 2p N > m \geq 2 $.

\subsection{Warm-up case: $g=2$, $X$ defines a polarization of type $(1,p)$ with $p$ a prime number.}
In this case $X$ is a $\ZZ/p$ \'etale covering of a genus $2$ curve $C =\Theta$, and $X/G= \PP^1$.

That the  canonical map of $X$ is not birational means that $X$ is hyperelliptic, and then $f : X \ra \PP^1$
factors through the quotient of the  hyperelliptic involution $\iota$, which centralizes $G$.
If we set $\Sigma := X / \iota \cong \PP^1$, we see that $\Sigma$ corresponds to a subgroup of order two of
$G$: since $\iota$ centralizes $G$, and is not contained in $\ZZ/p$, follows that $G$ is abelian, whence $p=2$.

We prove here a  result which might  be known (but we could not find it  in \cite{barth}; a referee points out that,
under the stronger assumption that $A$ is also general,
such a  result is proven in \cite{pirola} and in \cite{pp}).

\begin{lemma}
The general divisor in a linear system $|X|$ defining  a polarization of type $(1,2)$ is not hyperelliptic.
\end{lemma}
\begin{proof}
In this case of a polarization of type $(1,2)$,   the  linear system $|X|$ on $A$  is a pencil of genus $3$ curves
with $4$ distinct base points, as we now show.

We consider, as before, the inverse image $D$ of a  Theta divisor for $g=2$, which will be here called $C$, so that $A' = J(C)$.

The curve $D$ is hyperelliptic, being a Galois covering of $\PP^1$ with group $G = (\ZZ/2)^2$ (a so-called bidouble cover), fibre product of
two coverings respectively branched on $2,4$ points (hence $D = E \times_{\PP^1} \PP^1$, where $E$ is an elliptic curve,
and $C $ is the quotient of $D$ by the diagonal involution). The divisor cut by $H^0(A, \hol_A(X))$ on $X$
is the ramification divisor of $X \ra E$, hence it consists of $4$ distinct points.

The double covering $ D \ra E$ of an elliptic curve is also induced by the homomorphism $ J(D) \ra E$, with kernel isogenous
to $A'$. Therefore, moving $X$ in the linear system $|D|$, all these curves are a double cover of some elliptic curve $E'$,
such that $J(X)$ is isogenous to $A \times E'$.

Hence $ E' = X / \sigma$, where $\sigma$ is an involution.
 
Assume now that $X$ is hyperelliptic and denote by   $\iota$ the hyperelliptic involution: since $\iota$ is central,
$\iota$ and $\sigma$ generate a group $G \cong (\ZZ/2)^2$, and we claim that the quotient of the involution $\iota \circ \sigma$ 
yields a genus $2$ curve $C'$ as quotient of $X$.  In fact,  the quotient $X / G \cong \PP^1$, 
hence we have a bidouble cover of $\PP^1$ with branch loci of respective degrees $2,4$
and the quotient $C'$ is a double covering of $\PP^1$ branched in $6$ points.

Hence the hyperelliptic curves in $|D|$ are just the degree $2$ \'etale coverings   of genus $2$ curves, and are only a finite number 
in $|D|$ (since $J(C')$ is isogenous to $A$).
 \end{proof}

\subsection{We proceed with the  Key Step: showing that birationally $\Sigma$ must  lie between $X$ and $Y$.}
\subsubsection{To achieve this we need to recall some basic facts about covering spaces.} 

We shall use the Grauert-Remmert  \cite{g-r} extension of Riemann's theorem, stating that finite coverings $ X \ra Y$ between normal
varieties correspond to covering spaces $X^0 \ra Y^0$ between respective Zariski open subsets of $X$, resp. $Y$.

\begin{rem}
Given a connected unramified  covering $ X\ra Y$ between good spaces, we let $\widetilde{X}$ be the universal covering of $X$,
and we write  $ X =  H \backslash \widetilde{X}$,   $ Y =  \Ga \backslash \widetilde{X}$, hence $x \in X$ as $x =H \  \widetilde{x}$, $y \in Y$ as  $y = \Ga \ \widetilde{x} $, where both groups act on the left.

(1) Then the group of covering trasformations 
$$ G : = Aut (X \ra Y) \cong   N_H / H,$$
where $N_H$ is the normalizer subgroup of $H$ in $\Ga$.

(2) The monodromy group $Mon  (X \ra Y)$, also called the Galois group of the covering in \cite{debarre},
is defined as $\Ga / \core{H}$, where $\core{H}$ is the maximal normal subgroup of $\Ga$ contained in $H$,
$$ \core{H} = \cap_{\ga \in \Ga} H^{\ga} =  \cap_{\ga \in \Ga} (\ga^{-1} H {\ga} ),$$
 which is also called the normal core of $H$ in $\Ga$.

(3) The two actions of the two above groups on the fibre over $\Ga \widetilde{x}$, namely $H \backslash \Ga 
\cong \{ H \ga  \widetilde{x}| \ga \in \Ga \}$, commute, since $G$ acts on the left, and $ \Ga$ acts on the right
(the two groups coincide exactly when $H$ is a normal subgroup of $\Ga$, and we have what is called a normal covering space, or
a Galois covering).

We have in fact an antihomomorphism $$\Ga \ra Mon  (X \ra Y)$$ with kernel $\core{H}$.

(4) Factorizations of the covering $ X \ra Y$ correspond to intermediate subgroups $H'$ lying in between,
 $ H \subset H' \subset \Ga$.
\end{rem}

\subsubsection{Our standard situation}
We consider the composition of finite coverings $ X \ra Y \ra P$.

To simplify our notation, we consider the corresponding composition of unramified covering spaces
of Zariski open sets, and the corresponding fundamental groups 
$$  1 \ra K_1 \ra H_1 \ra \Ga_1.$$

Then the monodromy group of the Gauss map of $X$ is the quotient $\Ga_1 / \core{K_1}$.

We shall now divide all the above groups  by the normal subgroup $\core{K_1}$, and obtain 
$$  1 \ra K \ra H \ra \Ga, \ \Ga = Mon (X^0 \ra P^0).$$

Since $X\ra Y$ is Galois, and by Debarre's theorem $ Y \ra P$ is a generic covering, we have:

\begin{itemize}
\item
we have a surjection of the monodromy group $\Ga \ra \mathfrak S_N : =  \mathfrak S ( H \backslash \Ga)$  with kernel $\core{H}$,  
where $N = \left| H \backslash \Ga \right| = \frac{g!}{2}$. 
\item
$H$ is the inverse image of a stabilizer, hence $H $ maps onto $\mathfrak S_{N-1}$,
\item
$H$ is a maximal subgroup of $\Ga$, since $\mathfrak S_{N-1}$ is a maximal subgroup of $\mathfrak S_{N}$;
\item
$ H / K \cong G : = D_p$, 
\item
$\Ga$ acts on the fibre $M : = K \backslash \Ga$
\item
$K' : = K \cap \core{H}$  is normal in $K$, and $K'' : = K / K'$ 
maps isomorphically to  a normal subgroup of $\mathfrak S_{N-1}$, which is a quotient of $G$, hence it has index at most $2p$.
Therefore there are only two possibilities:
 \subitem
(I) $K'' = \mathfrak S_{N-1}$, or
\subitem
(II) $K'' = \mathfrak A_{N-1}$.
\end{itemize}

We consider now the case where there is a nontrivial factorization of the Gauss map $f$,
$ X \ra \Sigma \ra P$. 

Define $\widehat{H}$ to be the subgroup of the monodromy group $\Ga$ associated to $\Sigma$, so that
$$ 1 \ra K \ra \widehat{H} \ra \Ga,$$
and set:
$$   H' : =   \widehat{H} \cap \core{H}, \ H '' : = \widehat{H} / H' , \ H'' \subset   \mathfrak S_{N}.$$

Obviously we have $K'' \subset H''$: hence 

\begin{itemize}
\item
in case (I), where  $K''  = \mathfrak S_{N-1}$, we have either
\subitem
(Ia)  $H'' = \mathfrak S_{N-1}$ or
\subitem
(Ib)  $H'' = \mathfrak S_{N}$, while
\item
in case (II), where $K'' = \mathfrak A_{N-1}$,either
\subitem
(IIa) $H'' = \mathfrak A_{N-1}$,
\subitem
(IIb) $H'' = \mathfrak A_{N}$, or
\subitem
(IIc) $H'' = \mathfrak S_{N}$ 

\end{itemize}

We first consider  cases b) and c) where the index of $ H'' $ in $ \mathfrak S_{N}$ is $\leq 2$. Hence the index
of $\hat{H}$ in $\Ga$ is at most $2$ times the index of $H'$ in  $\core{H}$. Since $K' \subset H'$,
and  $ \core{H} / K'$ is a quotient of $ H/K$, we see that the index of $\hat{H}$ in $\Ga$ divides $2p$
( in fact, in case (IIb)   $K'' = \mathfrak A_{N-1}$, hence in this case the cardinality of  $ core(H) / K'$
equals $p$).

\begin{lemma}
Cases  (b) and (c),  where the degree $m$ of the covering $\Sigma \ra P$ divides $2p$,  are not possible.

\end{lemma}
\begin{proof}
Observe in fact that $m$ equals the index of $\widehat{H}$ in $\Ga$. We have two factorizations of 
the Gauss map $f$:
$$ X \ra Y \ra P , \ X \ra \Sigma \ra P,$$ hence we have that the degree of $f$ equals
$$ (2p) N = m ( N \frac{2p}{m}).$$

Consider now the respective ramification divisors $\sR_f,  \sR = \sR_{Y},  \sR_{\Sigma}$
of the respective maps $f : X \ra P$, $\Psi : Y \ra P$, $ \Sigma \ra P$.

Since $ X \ra Y$ is quasi-\'etale (unramified in codimension $1$), $\sR_f$ is the inverse image of $\sR$,
hence  $\sR_f$  maps to the branch locus $\sB$ with mapping degree $2p$.

Since the branch locus is known to be irreducible, and reduced, and $ \sR_{\Sigma}$ is non empty,
it follows that $ X \ra \Sigma$ is quasi-\'etale, and  $\sR_f$ is the inverse image of $ \sR_{\Sigma}$,
so that $\sR_f$  maps to the branch locus $\sB$ with mapping degree $ ( N \frac{2p}{m}) \de$,
where $\de$ is the mapping degree of $ \sR_{\Sigma}$ to $\sB$, which is at most $p$.

From the equality $ 2p = \de  N (\frac{2p}{m})$ and since $N = g! / 2$, we see that $g!$ divides $4p$,
which is absurd for $g \geq 4$. For $g=3$ it follows that $p=3$, and either $\de=2$, $ m = 6$,
or $\de = 1$, and $m=3$.

To show that  these special cases cannot occur, we can use several arguments.

For the case $m=3$, then $\Sigma$ is a nondegenerate surface of degree $3$ in $\PP^4$
(hence it is by the way the Segre variety $\PP^2 \times \PP^1$), and has a linear projection to $\PP^2$
of degree $3$: hence its branch locus is a curve of degree $6$, contradicting that $\sB$ has degree $12$
(see \cite{c-s}, $\sB$ is the dual curve of the plane quartic curve whose Jacobian is $A'$).

 For the case $m=6$, we consider the normalization of the fibre product $$ Z: = \Sigma \times_P Y,$$
so that there is a morphism of $X$ into $Z$, which we claim to be birational. In fact, the degree of
the map $Z \ra P$ is $18$, and each component $Z_i$ of $Z$  is a covering of  $\Sigma$; therefore, if  
$Z$ is not irreducible, there is a component $Z_i$ mapping  to $P$ with degree $6$: and the conclusion is that
$\Sigma \ra P$ factors through $Y$, which is what we wanted to happen, but assumed not to happen. 

If $Z$ is irreducible, then  $X$ is birational to $Z$, hence  the group $G$ acts on $Z$, and trivially on $Y$: 
follows that  the fibres of  $ X \ra \Sigma$
are $G$-orbits, whence $Y = \Sigma$, again a contradiction.

\begin{rem}
Indeed, we know (\cite{c-s}) that the monodromy group of $\Theta \ra P$ is $\mathfrak S_4$, 
and  one  could  describe explicitly $\Ga$  in relation to the series of inclusions 
$$ 1 \ra K \ra K_1 \ra H \ra \Ga,$$
where $K_1$ is the subgroup associated to $\Theta$.

At any rate, if $m=3$, we can take the normalization of the fibre product $W : = \Theta \times_P \Sigma$, and argue as before that 
if $W$ is reducible, then $\Theta$ dominates $\Sigma$, which is only possible if $\Sigma =Y$.

While, if $W$ is irreducible, $W = X$ and the fibres of $ X \ra \Sigma$ are made of $\ZZ/3$-orbits, hence again 
$\Theta$ dominates $\Sigma$.

\end{rem}

\end{proof}

Excluded cases (b) and (c), we are left with case (a), where  
$$H''  \subset  \mathfrak S_{N-1} \Rightarrow \widehat{H} \subset H,$$
equivalently $\Sigma$ lies between $X$ and $Y$, hence it corresponds to an intermediate subgroup
of $G = D_p$, either $\ZZ/p$ and then $\Sigma= \Theta$, or $\ZZ/2$, and then $\Sigma = X / \pm 1$
for a proper choice of the origin in the Abelian variety $A$.

\subsection{In the case where $\Sigma$ lies between $X$ and $Y$, for a general deformation of $X$
 the canonical map becomes birational.}
Here, we can soon dispense of the case $\Sigma= \Theta$: just using exactly the same argument 
we gave in lemma \ref{induction}, that  $\ZZ/p$ does not act on $H^0(X, \hol_X(K_X))$ as the identity,
so the fibres of $X \ra \Theta$ are separated.

For the case where $\Sigma = X / (\ZZ/2)$ we need to look at the vector space $H^0(X, \hol_X(K_X))$,
which is  a $p$-dimensional representation of $D_p$, where we have a basis $\theta_1, \dots, \theta_p$
of eigenvectors for the different characters of $\ZZ/p$. In other words, for a generator $r$ of $\ZZ/p$,
we must have
$$ r (\theta_i ) = \zeta^i \theta_i.$$
If $s$ is an element of order $2$ in the dihedral group, then $ r \circ s = s \circ r^{-1}$, hence
$$ r (s (\theta_i)) = s (r^{-1} (\theta_i)) = s ( \zeta^{-1}  \theta_i) =  \zeta^{-1} s (   \theta_i),$$
hence we may assume without loss of generality that 
$$ s (\theta_i) = \theta_{-i}, \ -i \in \ZZ/p.$$

It is then clear that $s$ does not act as the identity unless we are in the special case $p=2$.

 In the special case $p=2$ , we proceed as in \cite{c-s}. We have a basis $\theta_1, \theta_2$ of even functions,
i.e., such that $\theta_i (-z) = \theta_i(z)$, and $\Theta = X / (\ZZ/2)^2$, where $(\ZZ/2)^2$
acts sending $ z \mapsto \pm z + \eta$, where $\eta$ is a $2$-torsion point on $A$.
 Since the partial derivatives of $\theta_1$ are invariant for $ z \mapsto  z + \eta$,
and since  $\theta_2 (z + \eta) = - \theta_2 (z)$,  the canonical map  $\Phi_X$ factors through the involution $\iota : X \ra X$
such that
$$ \iota (z) =  -z + \eta.$$

If for a general deformation of $X$ as a symmetric divisor the canonical map would factor through $\iota$,
then $X$ would be $\iota$-invariant; being symmetric, it would be $(\ZZ/2)^2$-invariant, hence for all deformations
$X$ would remain the pull-back of a Theta divisor. This is a contradiction, since the Kuranishi family 
of $X$ has  higher dimension than the Kuranishi family  of a Theta divisor $\Theta$ (see \cite{c-s}).

\section{Embedding obstruction}  

\begin{thm}\label{divisorsonAV}
Let $X$ be an ample smooth  divisor in an Abelian variety $A$ of dimension $n+1$. 

Assume moreover that $X$ is not a hyperelliptic curve of genus $3$ yielding  a polarization of type $(1,2)$.

If $X$ defines a polarization of type
$(m_1, \dots, m_{n+1})$, then the canonical  map $\Phi_X$ of $X$ is a morphism, and it can be an embedding only
if $p_g(X) : = h^0(K_X)  \geq 2n+2$, which means that the Pfaffian $d : = m_1 \cdot m_2  \dots \cdot  m_{n+1}$ satisfies the
inequality $$ d \geq n + 2 = g + 1.$$
\end{thm}

\begin{proof}
Assume the contrary, $ d \leq n+1$, so that $X$ embeds in $\PP^{n + c}$, with codimension $ c \leq n$, $ c = d -1$.

Observe that the pull back of the hyperplane class of $\PP^{n+c}$ is the divisor class of $X$ restricted to $X$,
and that the degree $m$ of $X$ equals to the maximal self-intersection of $X$ in $A$, namely $m = X^{n+1} = d (n+1)!$. 

The Severi double point formula yields see (\cite{FL77}, also \cite{cat-og})
$$ m^2 =  c_n ( \Phi^* T_{\PP^{2n}} - T_X),$$ 
where $\Phi$ is the composition of $\Phi_X$ with a linear embedding $\PP^{n+c} \hookrightarrow \PP^{2n}$.

By virtue of the exact sequence $$ 0 \ra T_X \ra T_A | X \ra \hol_X(X) \ra 0, $$
we obtain
$$ m^2 = [(1 + X)^{2n+2} ]_n =  {2n+2 \choose n} X^{n+1} \Leftrightarrow d (n+1)! = m = {2n+2 \choose n} .$$

To have a quick proof, let us also apply the double point formula to the section of $X = \Phi (X)$ with a linear subspace of
codimension $(n-d+1)$, which is a variety $Y$ of dimension and codimension $(d-1)$ inside $\PP^{2d-2}$.

In view of the exact sequence 
$$ 0 \ra T_Y \ra T_A | Y \ra \hol_Y(X)^{n-d+2} \ra 0, $$

we obtain 
$$ m^2 = [(1+X)^{n + d +1} X^{n-d+1} ]_n = {n + d + 1 \choose d-1} X^{n+1}  $$
equivalently, 
$$ d (n+1)! = m = {n + d + 1 \choose d-1} .$$

Since, for $ d \leq n+1$,  $  {n + d + 1 \choose d-1}  \leq  {2n+2 \choose n}$, equality holding if and only if $ d= n+1$,
we  should have $ d= n+1$ and moreover 
$$ (n+1) (n+1)! = {2n+2 \choose n} \Leftrightarrow (n+2)! = {2n+2 \choose n+1} .$$ 

We have equality for $n=1$, but then when we pass from $n$ to $n+1$  the left hand side gets multiplied by $(n+3)$, the right hand side
by $\frac{(2n+4)(2n+3) }{ (n+2)^2}$, which is a strictly smaller number since 
$$(n+3)(n+2) = n^2 + 5n + 6 > 2(2n+3) = 4n + 6.$$ 

We   are done with showing the desired assertion since  we must have $n=1$, and $d=2$, and  in the case $n=1$, $d=2$  we  have a curve in $\PP^2$ of degree $4$ and genus $3$.

\end{proof}

\section{Remarks on the conjecture}

Recall Conjecture \ref{main}:

\begin{conj}
Assume that $(A,X)$ is a general pair of  a smooth  ample hypersurface $X\subset A$  in an Abelian variety $A$, giving a polarization with Pfaffian $d$
satisfying $ d \geq g + 1.$

Then  the canonical map $\Phi_X$ is an embedding of $X$.

\end{conj}

The first observation is that we can assume $g \geq 3$, since for a curve the canonical map is an embedding if and only if it is birational onto its image,
 hence we may apply here our main Theorem \ref{genbirat}. 

The second remark is that we have a partial  result which is similar to lemma \ref{induction}

\begin{lemma}
Assume that the polarization type $\overline{d} $ is divisible by
$\overline{\de} $. Then, if the embedding Conjecture \ref{main} holds true for $\overline{\de}$, and the linear system $|X'|$ is base point free
for the general element $X'$ yielding a polarization of type $\overline{\de}$, then the embedding conjecture also holds
true  for $\overline{d}$.
\end{lemma}
\begin{proof}
As in lemma \ref{induction} we reduce to the following situation:
 we have $\pi : X \ra X'$, which is an \'etale quotient with group $\ZZ/p$.  Moreover, since the  canonical divisor $K_X$
is the pull-back
of $K_{X'}$, the composition $\Phi_{X'} \circ \pi$ factors through $\Phi_X$, and a morphism $f : \Sigma \ra \Sigma'$.

Since by assumption  $\Phi_{X'}$ is an embedding, $\Phi_X$ is a local embedding at each point, and 
 it suffices to show that  $\Phi_X$ separates all the fibres.

 Recalling that
$$H^0 (X, \hol_X (K_X)) = \oplus_1^p H^0 (X', \hol_{X'} (K_{X'} + i \eta)),$$
(here $\eta$ is a nontrivial divisor of $p$-torsion), 
 this follows immediately if we know that for each point $p \in X'$ there are two distinct  eigenspaces which do not have $p$ as a base-point.
 
 Under our strong assumption $|K_{X'} + i \eta |$ contains the restriction of $|X' + i \eta |$ to $X'$,
 but $|X' + i \eta |$    is a translate of $|X'|$, so it is base-point free.
  
\end{proof}

Already in the case of surfaces ($n=2, g=3$) the result is not fully established, we want $ d \geq 4$, and the case of
a polarization of type $(1,1,4)$ is not yet written down (\cite{cesarano} treats the case of
a polarization of type $(1,2,2)$, which is quite interesting for the theory of canonical surfaces in $\PP^5$).

 Were our conjecture too optimistic, then the question would arise about  the exact  range of validity for the statement
of embedding of a general pair $(X,A)$.

\medskip

 {\bf Acknowledgements:} the first author would like to  thank Edoardo Sernesi and Michael L\"onne  for  interesting conversations.
 
 Thanks to the second referee for useful suggestions on how to improve the exposition.


\end{document}